\documentclass[paper=a4, fontsize=11pt]{scrartcl} 

\usepackage[T1]{fontenc} 
\usepackage{fourier} 
\usepackage[english]{babel} 
\usepackage{amsmath,amsfonts,amsthm} 
\usepackage{amsmath, calligra, mathrsfs}
\usepackage{amssymb}
\usepackage{mathtools}
\usepackage{hyperref}
\usepackage{mathdots}
\usepackage{array}
\usepackage[mathscr]{euscript}
\usepackage{verbatim}
\usepackage[dvipsnames]{xcolor}
\usepackage{mdframed} 
\usepackage{soulutf8}
\usepackage{stmaryrd}
\usepackage{tikz-cd}
\usepackage{hyperref}
\usepackage{lipsum} 

\usepackage{sectsty} 
\allsectionsfont{\centering \normalfont\scshape} 

\usepackage{fancyhdr} 
\usepackage{xurl}
\newcommand*{\sheafhom}{\mathcal{H}\kern -.5pt om}
\numberwithin{equation}{section} 
\numberwithin{figure}{section} 
\numberwithin{table}{section} 

\newtheorem{thm}{Theorem}[section]

\newtheorem{prop}[thm]{Proposition}

\newtheorem{lem}[thm]{Lemma}

\newtheorem{conj}[thm]{Conjecture}

\theoremstyle{definition}
\newtheorem{defn}[thm]{Definition}

\newtheorem{exmp}[thm]{Example}

\theoremstyle{remark}
\newtheorem{rem}[thm]{Remark}

\DeclareMathOperator{\St}{st}

\DeclareMathOperator{\lk}{lk}

\DeclareMathOperator{\Ast}{ast}

\newcommand{\horrule}[1]{\rule{\linewidth}{#1}} 

\title{	
	\normalfont \normalsize 
	\textsc{} \\ [25pt] 
	\horrule{0.5pt} \\[0.4cm] 
	\huge $f$-vectors of balanced simplicial complexes, flag spheres, and geometric Lefschetz decompositions
	
	\horrule{2pt} \\[0.5cm] 
}

\author{Soohyun Park \\ \href{mailto:soohyun.park@mail.huji.ac.il}{soohyun.park@mail.huji.ac.il} } 

\date{\normalsize October 11, 2024} 

\begin{document}
	
	\maketitle 
	
	
	\begin{abstract}
		\noindent We show that there are $f$-vectors of balanced simplicial complexes giving a source of simplicial complexes exhibiting a Boolean decomposition similar to a geometric Lefschetz decomposition. The objects we are working with are $h$-vectors of flag spheres and balanced simplicial complexes whose $f$-vectors are equal to them. This builds on work of Nevo--Petersen--Tenner on a conjecture of Nevo--Petersen that the gamma vector of an odd-dimensional flag sphere is the $f$-vector of a balanced simplicial complex (which was shown for barycentric subdivisions by Nevo--Petersen--Tenner). We can connect our decomposition to positivity questions on reciprocal/palindromic polynomials associated to flag spheres and geometric questions motivating them. In addition, we note that the degrees in the Lefschetz-like decomposition are not halved unlike the usual $h$-vector setting.
		
	\end{abstract}
	
	\section*{Introduction} 
	
	The objects which we will discuss are combinatorial objects with $f$-vectors that seem to have some kind of geometric structure. They come from flag spheres, which are simplicial spheres with special geometric properties. For example, they are both involved in questions about a special case of the Hopf conjecture for certain piecewise Euclidean manifolds in the sense of Gromov's CAT(0) inequality (\cite{CD}, Section 4.2.C on p. 122 of \cite{Gro}) and postivity properties of reciprocal/palindromic polynomials lying between unimodality and real-rootedness observed in many combinatorial settings \cite{Athgam}. \\

	The geometric question on nonpositive curvature can be translated to a combinatorial form using flag complexes, which are defined below.
	
	\begin{defn} (p. 270 of \cite{Gal}) \\
		A simplicial complex is \textbf{flag} if every minimal nonface has size 2. In other words, a subset $A \subset V(\Delta)$ of the vertex set $V(\Delta)$ is a face of $\Delta$ if and only if $e \in A$ for every 2-element subset $e \subset A$. Equivalently, it is the clique complex of a graph. \\
	\end{defn}
	
	\begin{conj} (Charney--Davis, Conjecture D on p. 118 and p. 127 of \cite{CD}) \label{cdconj} \\
		Let $d$ be an even number and $K$ be a $(d - 1)$-dimensional flag sphere. Then, we have that \[ (-1)^{ \frac{d}{2} } \sum_i \left( -\frac{1}{2} \right)^{i + 1} f_i \ge 0. \] Equivalently, we have that \[ (-1)^{ \frac{d}{2} } \sum_{i = 0}^d (-1)^i h_i \ge 0.  \]
	\end{conj}
	
	This is a special case of a conjecture involving the following invariant, which as originally considered in the context of real-rootedness of $h$-polynomials of flag spheres (some geometric motivation on p. 277 of \cite{Gal}). \\
	
	\begin{defn} (p. 272 of \cite{Gal}) \\
		The \textbf{gamma vector} of a reciprocal/palindromic polynomial $h(x)$ is the unique vector $\gamma = (\gamma_0, \ldots, \gamma_{ \frac{d}{2} })$ such that \[ h(x) = \sum_{i = 0}^{ \lfloor \frac{d}{2} \rfloor } \gamma_i x^i (x + 1)^{d - 2i}. \]
		
		In practice, we will take $d$ to be even since we're studying odd-dimensional flag spheres. \\
	\end{defn}
	
	\begin{conj} (Gal, Conjecture 2.1.7 on p. 272 of \cite{Gal}) \\
		The gamma vector of a flag sphere is nonnegative. The nonnegativity of the index $\frac{d}{2}$ component is Conjecture \ref{cdconj}. \\
	\end{conj}
	
	There has been a lot interesting work on this invariant from many different perspectives including early work of Foata--Sch\"utzenberger on permutation statistics and signatures of toric varieties in work of Leung--Reiner. A survey in a combinatorial context is given by Athanasiadis (see p. \cite{Athgam}) and an algebraic interpretation of the top component as the signature of a certain toric variety under certain conditions is given by Leung--Reiner \cite{LR}. \\
	
	The specific question we will study comes from work of Nevo--Petersen \cite{NP} and Nevo--Petersen--Tenner \cite{NPT} which suggests that the gamma vector is not only nonnegative but also the $f$-vector of some simplicial complex. \\
	
	\begin{defn} (Definition 4.1 on p. 95 of \cite{St}) \\
		Let $\Delta$ be a $(d - 1)$-dimensional simplicial complex on the vertex set $V$. We say that $\Delta$ is \textbf{balanced} if there is a map $\kappa : V \longrightarrow [d]$ such that if $\{ x, y \}$ is an edge of $\Delta$, then $\kappa(x) \ne \kappa(y)$. \\
		
		Equivalently, if we think of $\kappa$ as a \textbf{coloring} of $V$ with the colors $1, \ldots, d$, then every face of $\Delta$ has all its vertices colored differently. \\
	\end{defn}
	
	\begin{conj} (Nevo--Petersen, Nevo--Petersen--Tenner on p. 1365 of \cite{NPT}, Conjecture 1.4 on p. 504 of \cite{NP}) \\
		The gamma vector of a flag simplicial sphere is the $f$-vector of some balanced simplicial complex.
	\end{conj}
	
	\begin{thm} (Nevo--Petersen--Tenner, Theorem 1.1 on p. 1365 of \cite{NPT}) \\
		If $\Delta$ is the barycentric subdivision of a simplicial complex, then $\gamma(\Delta)$ is the $f$-vector of a balanced simplicial complex.
	\end{thm}
	
	Instead of the gamma vector itself, we will study properties of simplicial complexes $\Gamma$ such that $f(\Gamma) = h(\Delta)$ for a given flag sphere $\Delta$. This is motivated by the following result of Nevo--Petersen--Tenner \cite{NPT}:
	
	\begin{prop} (Nevo--Petersen--Tenner, proof of Proposition 6.2 on p. 1377 of \cite{NPT}) \\
		Let $\Delta$ be a flag sphere and $\Gamma$ be a simplicial complex such that $h(\Delta) = f(\Gamma)$ (which we can take to be balanced). If there is a $\left( \frac{d}{2} - 1 \right)$-dimensional simplicial complex $S$ of dimension $\le \frac{d}{2} - 1$ such that \[ \Gamma =   \{ F \cup G : F \in S, G \in 2^{[d - 2|F|]} \} \] with $[d]$ a vertex set disjoint from $S$, then $\gamma(\Delta) = f(S)$. \\
	\end{prop}
	
	Looking in this direction, we took a closer look at simplicial complexes $\Gamma$ such that $f(\Gamma) = h(\Delta)$ when $\Delta$ is a flag sphere and how they are constructed. Combining work of Caviglia--Constantinescu--Varbaro (Corollary 2.3 on p. 474 of \cite{CCV}) and Bj\"orner--Frankl--Stanley ($a = (1, \ldots, 1)$ case of Theorem 1 on p. 23, p. 30 -- 32, and Section 6.4 on p. 33 of \cite{BFS}), we see that $\Gamma$ can be taken to be a balanced simplicial complex (i.e. a $(d - 1)$-simplicial complex with a proper $d$-coloring of its 1-skeleton). Given a balanced simplicial complex $\Gamma$, we note that filling in colors that are missing from each face of $\Gamma$ can be used to construct a balanced shellable simplicial complex $\Delta$ with the an $h$-vector which is equal to the $f$-vector given balanced simplicial complex (p. 30 -- 32 \cite{BFS}). We note that the Boolean decomposition suggested by Nevo--Petersen--Tenner sort of looks similar to an ``incomplete'' completion of the unused colors with the vertex set $[d]$ denoting colors used. \\
	
	Going in this direction, we can show the following:
	
	\begin{thm}
		Given an even number $d$, the gamma vector of a $(d - 1)$-dimensional flag sphere is the $f$-vector of some simplicial complex of dimension $\le \frac{d}{2} - 1$ whose 1-skeleton has a proper coloring by $\le d$ colors. \\
	\end{thm}
	
	Our approach starts by looking at how the interpretation above works for 1-dimensional flag spheres (i.e. $n$-gons for $n \ge 4$). This made us set $[d]$ to be the vertex set of the unique facet of $\Gamma$ and $S$ to be the subcomplex of $\Gamma$ with vertices disjoint from $[d]$. Then, we use an induction on $\ell$ such $\dim \Delta = 2\ell - 1$ using a partition of unity argument involved in Adiprasito's proof of the $g$-conjecture (Lemma 3.4 on p. 14 of \cite{Adip}) to construct a Boolean decomposition of $\Gamma$ out of $\Gamma_e$ such that $f(\Gamma_e) = h(\lk_\Delta(e))$ (noting that $h(\lk_\Delta(e))$ is a flag sphere of dimension $d - 3$) and show it is compatible with the definition we made for the Boolean decomposition at the beginning. Note that this involves thinking about analogues of structure that occur in the 1-dimensional case. \\
	
	Finally, we note that the Boolean decomposition suggested by Nevo--Petersen--Tenner looks similar to the Lefschetz decomposition for compact K\"ahler manifolds and discuss connections between the balanced simplicial complexes we study and combinatorial examples in the literature (e.g. work of Kim--Rhoades \cite{KR}) on relationships between Lefschetz properties and Boolean lattices. Finally, we note that the usual coloring linear system of parameters for balanced simplicial complexes can be interpreted to have an exterior algebra-like structure in a certain sense. Since the degrees are \emph{not} halved unlike many combinatorial examples in the literature, it would be interesting to find a geometric source of the patterns we observed. \\
	
	
	\section*{Acknowledgments} 
	I would like to thank Karim Adiprasito for introducing me to this topic as part of a former group project.

	\section{Background} 
	
	Before entering the main content we will discuss, we mention some definitions and results which could be used in later sections. \\
	

	\begin{defn}
		A \textbf{reciprocal/palindromic} polynomial $h(t) = h_0 + h_1 t + \ldots + h_{d - 1} t^{d - 1} + h_d t^d$ of degree $d$ is one where $h_k = h_{d - k}$ for all $0 \le k \le \frac{d}{2}$. \\
	\end{defn}
	
	An important example from geometric combinatorics comes from $h$-vectors of boundaries of simplicial polytopes (p. 67 of \cite{St}). In our setting, we will look at generalized homology spheres (p. 271 of \cite{Gal}). \\

	
	In particular, the ones that we are interesting are \emph{flag} simplicial complexes. \\
	
	\begin{defn} (p. 270 of \cite{Gal}) \\
		A simplicial complex $\Delta$ is \textbf{flag} if every minimal nonface has size 2. In other words, a subset $A \subset V(\Delta)$ of the vertex set $V(\Delta)$ is a face of $\Delta$ if and only if $e \in A$ for every 2-element subset $e \subset A$. Equivalently, it is the clique complex of a graph. \\
	\end{defn}
	
	In addition to the combinatorial setting that we're working in, the flag property has been studied in the context of nonpositive curvature of certain cubical complexes in the sense of Gromov's CAT(0) inequality (Conjecture D on p. 118 and 127 of \cite{CD}, an alternate form on p. 135 of \cite{CD}). \\
	
	
	
	Next, we mention constructions and properties of flag spheres (and flag simplicial complexes in general) which will be used for studying local-global properties and inductive arguments involving the dimension. \\
	

	\begin{defn} (p. 60 of \cite{St}, Definition 5.3.4 on p. 232 of \cite{BH})
		Given a simplicial complex $\Delta$, its \textbf{link} over a face $F \in \Delta$ is the subcomplex \[ \lk_\Delta(F) \coloneq \{ G \in \Delta : F \cup G \in \Delta, F \cap G = \emptyset \}. \] 
		
		These are the faces of $\Delta$ formed by removing $F$ from faces of $\Delta$ containing $F$. \\
	\end{defn}
	
	We first list a special local property of flag spheres before listing the codimension 2 subcomplexes which will be used for inductive arguments. \\
	
	\color{black}
	
	
	\begin{prop}
		Given a flag simplicial complex $\Delta$, we have that $\lk_\Delta(p) \cap \lk_\Delta(q) = \lk_\Delta(p, q)$ if $(p, q) \in \Delta$ is an edge of $\Delta$. More generally, we have that $\lk_\Delta(P) \cap \lk_\Delta(Q) = \lk_\Delta(P \cup Q)$ if $P \cup Q \in \Delta$. Note that the inclusion $\lk_\Delta(p) \cap \lk_\Delta(q) \subset \lk_\Delta(p, q)$ does not hold for arbitrary simplicial complexes.
	\end{prop}

	\begin{prop}
		Given a flag homology sphere $\Delta$, a link $\lk_\Delta(e)$ over an edge $e \in \Delta$ is also a flag homology sphere. \\
	\end{prop}
	
	Flagness and being a homology sphere is preserved under taking links over vertices (Definition 1.2.1 on p. 271 of \cite{Gal} and Theorem 5.1 on p. 65 of \cite{St}). \\
	
	\color{black}
	
	The main invariant we are concerned with is the gamma vector, which is defined below. \\
	
	\begin{defn} (p. 272 of \cite{Gal}) \\
		The \textbf{gamma vector} of a reciprocal/palindromic polynomial $h(x)$ is the unique vector $\gamma = (\gamma_0, \ldots, \gamma_{ \frac{d}{2} })$ such that \[ h(x) = \sum_{i = 0}^{ \lfloor \frac{d}{2} \rfloor } \gamma_i x^i (x + 1)^{d - 2i}. \]
		
		In practice, we will take $d$ to be even since we're studying odd-dimensional flag spheres. \\
	\end{defn}
	
	In a geometric context, it was studied by Gal \cite{Gal} who was studying real-rootedness properties of $h$-vectors of flag spheres (with some motivation on p. 277 of \cite{Gal}). The positivity of this vector lies somewhere between unimodality and real-rootedness. It has also had interesting applications in many different parts of algebraic combinatorics (which is surveyed in \cite{Athgam}). \\
	
	Finally, we list some properties of Cohen--Macaulay simplicial complexes which we may use later. \\
	
	\begin{thm} (Theorem 5.10 on p. 35, proofs of Lemma 2.4 and Theorem 2.5 on p. 81 -- 82 and p. 82 -- 83, Theorem 4.5 on p. 98 -- 99 of \cite{St}) \label{basissoc} \\
		\begin{enumerate}
			\item (Theorem 5.10 on p. 35 of \cite{St}) \\
			Let $R$ be an $\mathbb{N}^m$-graded $k$-algebra and $M$ be a $\mathbb{Z}^m$-graded $R$-module. Suppose that $M$ is a $(d - 1)$-dimensional Cohen--Macaulay $R$-module (Definition 5.8 on p. 35 of \cite{St}), with a homogeneous system of parameters $\theta = (\theta_1, \ldots, \theta_d)$ (i.e. homogeneous elements forming a regular sequence). Let $\eta_1, \ldots, \eta_s$ be a collection of homogeneous forms in $M$. \\
			
			This has an analogue to the existence of monomials $\eta_j = x^{H_j}$ such that $\eta_1, \ldots, \eta_s$ forms a basis for $k[T]$ as a $k[\theta]$-module if $T$ is a Cohen-Macaulay simplicial complex (a sort of graded Noether normalization -- see Theorem 1.5.17 on p. 37 -- 38 of \cite{BH}, Theorem 5.10 on p. 35 and application in Theorem 2.5 on p. 82 of \cite{St}). \\
			
			Then \[ M = \bigoplus_{i = 1}^s \eta_i k[\theta] \] if and only if $\eta_1, \ldots, \eta_s$ is a $k$-basis for $M/\theta M$. \\
			
			For such a choice of $\theta$ and $\eta$, it follows that \[ F(M, \lambda) = \frac{ \sum_{i = 1}^s \lambda^{\deg \eta_i} }{ \prod_{j = 1}^d (1 - \deg \lambda^{\deg \theta_j}) }. \]
			
			\item (Lemma 2.4 and Theorem 2.5 on p. 81 -- 83, Theorem 4.5 on p. 98 of \cite{St}) \\ 
			Let $\Delta$ be a $(d - 1)$-dimensional Cohen--Macaulay simplicial complex. Setting $M = k[\Delta]$ and $\theta = (\theta_1, \ldots, \theta_d)$ to be a linear system of parameters, Part 1 implies that $k[\Delta]$ is a free module over $k[\theta]$ with basis $\eta_1, \ldots, \eta_s$ if and only if $\eta_1, \ldots, \eta_s$ is a $k$-basis for the Artinian reduction $A^\cdot(\Delta) \coloneq k[\Delta]/\theta k[\Delta]$ of the Stanley--Reisner ring $k[\Delta]$. \\
			
			When $\Delta$ is shellable, we can set $\eta_i = x^{r(F_i)}$ from the restrictions $r(F_i) = \Delta_i \setminus \Delta_{i - 1}$ of the facets $F_i$ of $\Delta$ (p. 79 of \cite{St}) yielding a partition \[ \Delta = [r(F_1), F_1] \cup \cdots \cup [r(F_s), F_s]. \]
			Note that the proof of this is based on showing that $x^{r(F_i)}$ is in the socle of $A^\cdot(\Delta)$ (p. 50 of \cite{St}). \\
			
			In the proof of Theorem 4.5 on p. 98 of \cite{St}, it is claimed that we can take $\eta_1, \ldots, \eta_s$ in Part 1 to be some set of monomials in the vertices of $\Delta$ and the balanced condition only seems to be used in the choice of linear system of parameters $\theta$ involved there. The reason for the use of monomials seems to be that the monomials $x^F$ for faces $F \in \Delta$ span the Artinian reduction $A^\cdot(\Delta)$ as a vector space (Lemma 2.4 on p. 81 of \cite{St}) and a maximal linearly independent subset of them forms a basis. \\
			
			\item (Lemma 2.4 on p. 81 of \cite{St}) \\
			Let $k[\Delta]$ be the Stanley--Reisner ring of Krull dimension $d$, and let $\theta_1, \ldots, \theta_d \in k[\Delta]_1$. Then $\theta_1, \ldots, \theta_d$ is a linear system of parameters for $k[\Delta]$ if and only if for every face $F \in \Delta$ (equivalently for every facet $F \in \Delta$), the restrictions $\theta_1|_F, \ldots, \theta_d|_F$ span a vector space of dimension equal to $F$. In particular, this means that a linear system of parameters for a simplicial complex $\Delta$ can be used to form one for subcomplexes via restriction of the linear forms to vertices that appear in the subcomplex. \\
			
		\end{enumerate}
	\end{thm}

	\section{$f$-vectors of balanced simplicial complexes and Boolean decompositions}

	\subsection{Boolean decompositions and $f$-vectors of balanced simplicial complexes}

	By work of Caviglia--Constantinescu--Varbaro (Corollary 2.3 on p. 474 of \cite{CCV}), the $h$-vector of any flag Cohen--Macaulay simplicial complex $\Delta$ is equal to the $h$-vector of some Cohen--Macaulay balanced simplicial complex $T$. Combining this with work of Bj\"orner--Frankl--Stanley ($a = (1, \ldots, 1)$ case of Theorem 1 on p. 23, Section 6.4 on p. 33 of \cite{BFS}), it is equal to the $f$-vector of a simplicial complex $\Gamma$ whose 1-skeleton is $d$-colorable. If $h_d(\Delta) \ne 0$ (e.g. $\Delta$ a flag sphere), then $\Gamma$ is a $(d - 1)$-dimensional balanced simplicial complex. We record this result below. \\
	
	\begin{thm} (Caviglia--Constantinescu--Varbaro and Bj\"orner--Frankl--Stanley, Corollary 2.3 on p. 474 of \cite{CCV} and $a = (1, \ldots, 1)$ case of Theorem 1 on p. 23, Section 6.4 on p. 33 of \cite{BFS}) \label{flagcmhtof} \\
		The $h$-vector of a $(d - 1)$-dimensional flag Cohen--Macaulay simplicial complex $\Delta$ is the $f$-vector of a simplicial complex $\Gamma$ whose 1-skeleton is $d$-colorable. If $h_d(\Gamma) \ne 0$ (e.g. $\Delta$ a flag sphere), we can take $\Gamma$ to be a $(d - 1)$-dimensional balanced simplicial complex. \\
	\end{thm}

	Keeping this in mind, we can look at the proposed interpretation of the gamma vector of a flag sphere. \\
	
	\begin{prop} (Nevo--Petersen--Tenner, proof of Proposition 6.2 on p. 1377 of \cite{NPT}) \label{gamdec} \\
		Let $\Delta$ be a flag sphere and $\Gamma$ be a simplicial complex such that $h(\Delta) = f(\Gamma)$ (which we can take to be balanced). If there is a $\left( \frac{d}{2} - 1 \right)$-dimensional simplicial complex $S$ of dimension $\le \frac{d}{2} - 1$ such that \[ \Gamma =   \{ F \cup G : F \in S, G \in 2^{[d - 2|F|]} \} \] with $[d]$ a vertex set disjoint from $S$, then $\gamma(\Delta) = f(S)$. \\
	\end{prop}
	
	This is reminiscent of the construction of a balanced shellable simplicial complex $\mathcal{C}(\Gamma)$ such that $h(\mathcal{C}(\Gamma)) = f(\Gamma)$ on p. 30 -- 32 of \cite{BFS}, which fills in missing colors of each face in a balanced simplicial complex. This decomposition can be considered as a sort of ``partial filling''. \\
	
	Another instance where a coloring complex was associated to the gamma vector of a flag sphere was our earlier work on Tchebyshev subdivisions and $f$-vectors of flag spheres (signed unused coloring complexes from \cite{Pflagdec}). \\
	

	\begin{rem} \textbf{(Uniqueness and vanishing conditions)} ~\\
		\vspace{-3mm}
		\begin{enumerate}
			\item Recall that the gamma vector $\gamma = (\gamma_0, \ldots, \gamma_{\frac{d}{2}})$ is uniquely defined for each reciprocal/palindromic polynomial $h$ of degree $d$. It is defined as the coefficients $\gamma_i$ such that \[ h(x) = \sum_{i = 0}^{ \frac{d}{2} } \gamma_i x^i (x + 1)^{ d - 2i }. \]
			
			The polynomials $x^i (x + 1)^{ d - 2i }$ for $0 \le i \le \frac{d}{2}$ give basis for reciprocal/palindromic polynomials of degree $d$. Since the vertex set of $S$ is disjoint from $[d]$ in the decomposition above, it yields the expansion \[ f_{i - 1}(\Gamma) = h_i(\Delta) = \sum_{ j = 0 }^i \gamma_j \binom{d - 2j}{i - j}. \] This gives an invertible linear transformation writing the $h$-vector components as a linear combination of those of the gamma vector. Since the $h_i$ for $0 \le i \le \frac{d}{2}$ and $\gamma_j$ yield basis elements for the palindromic/reciprocal polynomials of degree $d$, the $\gamma_j$ above must be the same as the $\gamma_i$ in the first expansion. \\

			\item  Since the only term of degree $d$ is in the index $i = 0$, we have that the leading coefficient $h_d = h_0$ is given by $\gamma_0$. In particular, the leading coefficient of $h(x)$ is nonzero if and only if $\gamma_0 \ne 0$. In our setting, we will have $h_0 = h_d = 1$ and $\gamma_0 = 1$. \\
			
			By work of D'Al\`i--Juhnke-Kubitzke--K\"ohne--Venturello \cite{DJKKV}, there is a regime of the Erd\H{o}s--R\'enyi model for random graphs where the gamma vector associated the boundary of the symmetric edge polytope of $G$ asymptotically almost surely has $\gamma_\ell = 0$ for all $\ell \ge 1$ (Theorem B on p. 489 of \cite{DJKKV}). In the Boolean decomposition, the $S$ has $\dim S \le \frac{d}{2} - 1$ and this implies that there are examples where its dimension is strictly smaller than $\frac{d}{2} - 1$. \\
		\end{enumerate}

	\end{rem}
	
	\subsection{1-dimensional case and local to global constructions}
	

	We will start thinking about Boolean decompositions and flag spheres by starting out with the 1-dimensional example. Afterwards, we will use a local-global construction to carry out an induction on $k$ such that the dimension of the odd-dimensional flag sphere is $2k - 1$. Note that the inductive step will also make use of an interpretation of the 1-dimensional example. \\

	\begin{exmp} \textbf{(1-dimensional case and a color palette) \\} \label{dim1case}
		In the 1-dimensional case ($d = 2$), the flag spheres come from $n$-gons with $n \ge 4$. Since $f_0 = n$ and $f_1 = n$, this means that $h_0(\Delta) = 1$, $h_1(\Delta) = n - 2$, and $h_2 = 1$. Substituting this into the gamma vector definition, we have that $\gamma_0 = 1$ and $\gamma_1 = n - 4$. \\
		

		Using the information above, we can consider how the gamma vector is related to a Boolean decomposition. The 1-dimensional simplicial complex $\Gamma$ consisting of $n - 2$ vertices with a single pair of vertices among them connected by an edge has $h(\Delta) = f(\Gamma)$. Label the vertices connected by an edge by $x_1$ and $x_2$ and the remaining $n - 4$ disjoint vertices by $x_3, \ldots, x_{n - 2}$. In terms of the Boolean decomposition, we can take the simplicial complex $S$ to be $0$-dimensional simplicial complex coming from the $n - 4$ disjoint vertices $x_3, \ldots, x_{n - 2}$. Taking the empty face $\emptyset \in S$ and $G \in 2^{[2]}$ corresponds to the faces of $\Gamma$ contained in the edge $\{ x_1, x_2 \} \in \Gamma$. If we take a vertex $x_i$ with $3 \le i \le n - 2$ contained in $S$, we have an element of $G \in 2^{[0]} = \emptyset$. This corresponds to the $n - 4$ individual vertices $x_i$ with $3 \le i \le n - 2$. \\
		
		Note that the 1-dimensional simplicial complex $\Gamma$ above (which uses $d = 2$) is balanced. If we take the interpretation above and use the $(k - 1)$-dimensional faces of $S$ to be the subcollection of $(k - 1)$-dimensional faces of $\Gamma$ that use the last $2k$ colors $d - 2k + 1, \ldots, d$, our decomposition means that the possible colors of $x_3, \ldots, x_{n - 2}$ can be any subset of $[2]$ and the vertices $x_1$ and $x_2$ (which are connected by an edge in $\Gamma$) formally correspond to colors. The former is consistent with none of the vertices $x_3, \ldots, x_{n - 2}$ being connected to each other by edges. In other words, using either of the two colors for each of the vertices $x_3, \ldots, x_{n - 2}$ would be compatible with the coloring interpretation of the Boolean decomposition. \\

		In general, the construction of a Boolean decomposition of a balanced $(d - 1)$-dimensional simplicial complex $\Gamma$ seems to involve a simplicial complex $S$ of dimension at most $\left( \frac{d}{2} - 1 \right)$ whose 1-skeleton is $d$-colorable. By a result of Bj\"orner--Frankl--Stanley ($a = (1, \ldots, 1)$ case of Theorem 1 on p. 23 and Section 6.4 on p. 33 of \cite{BFS}), this is equivalent to having $f(S) = h(B)$ for some $(d - 1)$-dimensional balanced Cohen--Macaulay (or shellable) simplicial complex $B$ with $h_i(B) = 0$ for $\frac{d}{2} + 1 \le i \le d$. \\
		
	\end{exmp}
	
	
	To construct the Boolean decomposition of a fixed balanced simplicial complex $\Gamma$ such that $f(\Gamma) = h(\Delta)$, we will start by taking the vertex set $[d]$ in the Boolean decomposition to be the vertices of the unique facet of $\Gamma$. Note that the facet is unique since $f_{d - 1}(\Gamma) = h_d(\Delta) = 1$. This is taking a page from the 1-dimensional example above. The simplicial complex $S$ will be taken to be the subcomplex of $\Gamma$ consisting of faces which do \emph{not} use any vertices from $[d]$. We would like to show that $\dim S \le \frac{d}{2} - 1$ and that the $F \cup G$ from the proposed decomposition give all the faces of $\Gamma$. \\
	
	Next, we induct on $k$ such that the dimension of the given flag sphere is $2k - 1$ using a local to global extension result used in the proof of the $g$-conjecture. \\
	
	\begin{lem} \textbf{(Partition of unity)} (Adiprasito, Lemma 3.4 on p. 14 -- 15 of \cite{Adip}) \label{partun} \\
		Let $\Delta$ be a $(d - 1)$-dimensional Cohen--Macaulay simplicial complex in $\mathbb{R}^d$. Then, for every $k < d$, we have a surjection \[ \bigoplus_{p \in V(\Delta)} A^k(\lk_\Delta(p)) \twoheadrightarrow A^k(\Delta). \]
		
		Dually, this is equivalent to an injection \[ A^k(\Delta) \hookrightarrow \bigoplus_{p \in V(\Delta)} A^k(\lk_\Delta(p)). \]
	\end{lem}

	
	Since $\Delta$ is flag, we have that $\lk_{\lk_\Delta(p)}(q) = \lk_\Delta(p) \cap \lk_\Delta(q) = \lk_\Delta(p, q)$ if $(p, q) \in \Delta$. The inclusion $\lk_\Delta(p) \cap \lk_\Delta(q) \subset \lk_\Delta(p, q)$ uses flagness and does not hold for arbitrary simplicial complexes. \\
	
	Given $0 \le k \le d - 2$, this means that we have a composition of injections \[ i : A^k(\Delta) \xhookrightarrow{\varphi} \bigoplus_{p \in V(\Delta)} A^k(\lk_\Delta(p)) \xhookrightarrow{\bigoplus_{p \in V(\Delta)} \psi_p} \bigoplus_{p \in V(\Delta)} \left( \bigoplus_{q \in V(\lk_\Delta(p))} A^k(\lk_\Delta(p, q)) \right) \cong \left( \bigoplus_{ \substack{ e \in \Delta \\ |e| = 2 } } A^k(\lk_\Delta(e)) \right)^{\oplus 2}. \]
	
	Note that the $\lk_\Delta(e)$-component of the partition of unity map is a sort of projection/restriction map which is basically the identity map on $x^H$ for each degree $k$ element of some fixed basis of $A^\cdot(\Delta)$ when $H \in \St_\Delta(e)$ (noting that $A^k(\lk_\Delta(e)) \cong A^k(\St_\Delta(e))$) and $0$ otherwise (see Lemma 2.4 on p. 81 of \cite{St}). This is because the partition of unity map is constructed out of the Cech complex coming from restrictions to stars of faces of $\Delta$. Also, an application of restrictions to subcomplexes of decompositions of Stanley--Reisner rings of simplicial complexes as free modules over $k[\theta]$ is given on p. 98 of \cite{St}. \\
	
	Using the fact that $h$-vectors of flag spheres $\Delta$ are $f$-vectors of balanced simplicial complexes $\Gamma$ (Theorem \ref{flagcmhtof}), we can give an interpretation of the decomposition above. Since $h(\Delta) = f(\Gamma)$, we note that $f_{k - 1}(\Gamma) = h_k(\Delta) = \dim A^k(\Delta)$. In particular, this means that $\dim B^k(\Gamma) = \dim A^k(\Delta)$ as vector spaces, where \[ B^\cdot(\Gamma) \coloneq \frac{k[\Gamma]}{(x_1^2, \ldots, x_n^2)} \] and $B^k(\Gamma)$ denotes the degree $k$ elements of $B^\cdot(\Gamma)$. \\ 
	
	Note that the construction of a balanced shellable simplicial complex $\mathcal{C}(\Gamma)$ such that $h(\mathcal{C}(\Gamma)) = f(\Gamma)$ on p. 30 -- 32 of \cite{BFS} induces a shelling of $\mathcal{C}(\Gamma)$ where the $(k - 1)$-dimensional faces $H \in \Gamma$ induce basis elements $x^H$ of $A^k(\mathcal{C}(\Gamma)) \cong A^k(\Delta)$ (from $h_k(\mathcal{C}(\Gamma)) = f_{k - 1}(\Gamma) = h_k(\Delta)$). The reason is that the faces $H \in \Gamma$ are restrictions (p. 79 of \cite{St}) of facets of $\mathcal{C}(\Gamma)$, which are constructed by inserting new vertices corresponding to the colors in $[d]$ which are unused by a given face of $\Gamma$. \\
	
	For each edge $e \in \Delta$, let $\Gamma_e$ be a balanced simplicial complex such that $f(\Gamma_e) = h(\Gamma_e)$. This means that the partition of unity maps compose to the following injection:
	
	\[ i : B^k(\Gamma) \hookrightarrow \left( \bigoplus_{ \substack{ e \in \Delta \\ |e| = 2 } } B^k(\Gamma_e) \right)^{\oplus 2} \]

	Aside from this, we will say a bit more about how we put together the Boolean decompositions of $\Gamma_e$. The set $[d - 2]$ corresponds to a $(d - 3)$-face of $\Gamma_e$ (the unique facet of $\Gamma_e)$. If we use compression complexes (see p. 263 of \cite{Zie} and a colored version on p. 1366 -- 1367 of \cite{NPT}) to form $\Gamma$, the image of the component $i_e : B^k(\Gamma) \longrightarrow B^k(\Gamma_e)^{\oplus 2}$ associated $e \in \Delta$ is defined as follows:

	\begin{itemize}
		\item  If we take $\Gamma$ to be the compression complex for $h(\Delta)$, its $(k - 1)$-dimensional faces are the first $h_k(\Delta)$ elements of $\binom{ \mathbb{N} }{k}$ (in reverse lexicographical order). This can be used to $\Gamma_e$ interpret as a subcomplex of $\Gamma$. Note that other vertex orderings can also be used. 
		
		\item Since $f(\Gamma) = h(\Delta)$, we can use this to fix a bijection between the $(k - 1)$-dimensional faces of $\Gamma$ and basis elements $x^H$ of $A^k(\Delta)$. 
		
		\item The $e$-component of the map $i$  is the identity map if the corresponding face of $\Delta$ belongs to $\lk_\Delta(e)$. Otherwise, it is mapped to 0 (see Theorem \ref{basissoc}).
		
		\item This map can be considered to be starting with precomposition by a sort of indicator function map before applying the projection map to the first $h_k(\lk_\Delta(e))$ components. The projection is analogous to the identity map $x^H \mapsto x^H$ when $H \in \Delta$ being precomposed with the indicator function (whether a face belongs to $\lk_\Delta(e)$ or not) on the original partition of unity map involving Artinian reductions. 
		
		\item The idea is to think about a cover of faces of $\Gamma$ by subcomplexes $\Gamma_e$ where $\Gamma$ and $\Gamma_e$ are analogous to $\Delta$ and $\lk_\Delta(e)$ in the partition of unity map. By a ``cover'', we mean that every face of $\Gamma$ is a face of $\Gamma_e$ for some $e$. To be more precise, consider the surjective form of the partition of unity map and think about the basis elements of $A^k(\lk_\Delta(e))$ as faces of a subcomplex $\Gamma_e \le \Gamma$ corresponding to subcomplexes forming a cover of $\Gamma$. Then, dualize the map to get the injective form.
	\end{itemize}

	While the map is written differently, the point is that this map is the identity map when it is nonzero as in the case of the $e$-component $i_e : A^k(\Delta) \longrightarrow A^k(\lk_\Delta(e))$ of the original partition of unity map. This can be called a projection since $h_k(\lk_\Delta(e)) \le h_k(\Delta)$ for any Cohen--Macaulay simplicial complex $\Delta$. Since $\Gamma$ is a simplicial complex, any $m$-element subset contained in an $k$-element set contained in $\Gamma$ also belongs to $\Gamma$. Note that the unique facet of $\Gamma_e$ is written as $[d - 2]$ for each edge $e \in \Delta$ when we use compression complexes. For enumerative purposes (e.g. what the gamma vector is equal to), what we need from the Boolean decomposition only seems to depend on the following parameters: \\

	\begin{itemize}
		\item The size $D$ of the Boolean part $B$ of the vertex set $V(\Gamma)$ \\
		
		\item $\dim S \le \frac{D}{2} - 1$ \\
		
		\item $B$ being disjoint from $V(S)$ \\
		
		\item The condition $|G| \le D - 2|F|$ for each face $F \cup G$ with $F \in S$ and $G \in 2^B$. \\ 
		
		\item The same thing applies to a correspondence with pairs $(P, Q)$ of ``independent'' sets $P$ and $Q$ satisfying the properties above for $F$ and $G$ where $P$ and $Q$ are not necessarily the same kind of object. \\
	
	\end{itemize}

	By our induction assumption, we have the decomposition \[ \Gamma_e = \{ F_e \cup G_e  : F_e \in S_e, G_e \in 2^{[d - 2|F_e| - 2]} \}, \] where $[d - 2]$ corresponds to the vertices of the unique facet of $\Gamma_e$ and $S_e$ is the subcomplex of $\Gamma_e$ consisting of faces whose vertices do \emph{not} come from $[d - 2]$ and $\dim S_e \le \frac{d}{2} - 2$. The $(d - 2)$-element sets corresponding to $[d - 2]$ for $\Gamma_e$ come from the unique facet of each $\Gamma_e$. In the context of a Boolean decomposition for $\Gamma$, we will insert $[d - 2]$ inside the set $[d]$ whose elements correspond to the vertex set of the unique facet of $\Gamma$. In addition, the condition on the Boolean part $G_e$ is compatible with the proposed Boolean decomposition \[ \Gamma = \{ F \cup G : F \in S, G \in 2^{[d - 2|F|]} \} \] of $\Gamma$ since a subset of $[d - 2|A| - 2]$ is a subset of $[d - 2|A|]$ for any set $A$. Finally, any collection of vertices which is disjoint from $[d]$ is also disjoint from $[d - 2]$ since $[d - 2] \subset [d]$. \\
	
	Note that the partition of unity map applies to $0 \le k \le d - 2$. The generators of these parts consist of faces $H \in \Gamma$ of dimension $\le d - 3$. For these parts, we can make use of the $F_e \cup G_e$ representation of $H$ on the component corresponding to an edge $e \in \Delta$. Since the partition of unity map is still an injection after decreasing the degree to a smaller positive number, there is a unique $F \in \Gamma$ mapping to a given tuple $(F_e)_{ \substack{ e \in \Delta \\ |e| = 2 } }$ and a unique $G \in \Gamma$ mapping to a given tuple $(G_e)_{ \substack{ e \in \Delta \\ |e| = 2 } }$. We can set these $F$ and $G$ to be our initial and Boolean components respectively. If the map is simply trivial, it respects the inclusions. If we use the compression complex for everything, the coordinates on for each edge $e \in \Delta$ are the same. In this case, the inclusion wouldn't be a literal identity map but a correspondence. For example, different $(d - 3)$-dimensional faces of $\Gamma$ would come from faces of different $\Gamma_e$. Alternatively, we can use a different ordering for the compression complex to get something close to a literal inclusion map instead of an injection. \\
	
	As mentioned earlier, the point is that we can think about this as corresponding a cover of $\Gamma$ by subcomplexes $\Gamma_e$. Since a subset of a face $H \in \Gamma_e$ is also a face of $\Gamma_e$, the $F_e$ and $G_e$ are nonzero at the components $e \in \Delta$ where $H$ is nonzero. To take $G_e$ as part of the ``Boolean'' part after the embedding, we can use the compression complexes for $\Gamma_e$ using vertex orders (for reverse lexicographical order) where the first $d - 2$ elements are given by different $(d - 2)$-element vertex subsets of the unique facet $[d] \in \Gamma$ but repeated pairs of vertices in different vertices have the same order. This is similar to excluding colors not used by $\Gamma_e$ covering $\Gamma_e$ when we're given a particular $d$-coloring of the vertices of $\Gamma$. The main thing to note is that $F_e$ and $G_e$ can be taken as is since the induced partition of unity map for the compression complexes is still the identity on components where it is nonzero and 0 on the other ones. Note that there are still $2^{d - 2|F| - 2}$ choices of $G$ for a given $F$ since the map is injective. We can also think about the pairing comment on $(P, Q)$ above. Taking compression complexes, we can take the $(d - 4)$-skeleton of each $\Gamma_e$ (i.e. everything except for the unique facet) to be a face of the $(d - 2)$-skeleton of $\Gamma$ excluding the unique facet of $\Gamma$. In other words, we will take $\Gamma_e$ to be a collection of subcomplexes $\Gamma_e \le \Gamma$ where each face of $\Gamma$ of dimension $\le d - 3$ is a face of $\Gamma_e$ for some $e \in \Delta$. For each such face of $\Gamma$, the injective form of the partition of unity map says that there is a unique tuple of such $\Gamma_e$ associated to each face $H \in \Gamma_e$. \\
	
 	The remaining dimensions of faces of $\Gamma$ are $d - 2$ and $d - 1$. As mentioned previously, there is a unique facet of $\Gamma$ and we set its vertex set to be the set $[d]$ in the Boolean decomposition. The $(d - 2)$-dimensional faces of $\Gamma$ are formed by adding a vertex to a $(d - 3)$-dimensional face of $\Gamma$. Adding a vertex to either of the initial or the Boolean part is compatible with the Boolean decomposition condition for $\Gamma$. If the vertex $p$ is added to the Boolean part to get $G = G_e \cup p$, we have that \[ |G_e| \le d - 2|F_e| - 2 \Longrightarrow |G| = |G_e| + 1 \le d - 2|F_e| - 1 < d - 2|F_e| = d - 2|F| \] since the initial part stays the same. If it is added to the initial part to get $F = F_e \cup p$, we have that \[ |G| = |G_e| \le d - 2|F_e| - 2 = d - 2(|F_e| + 1) = d - 2|F| \] since the Boolean part stays the same. Thus, the size condition on the Boolean part is satisfied in both cases. If a $(d - 3)$-dimensional face of $\Gamma$ belongs to a particular $\Gamma_e$, a new vertex of $\Gamma$ added to it forming a $(d - 2)$-face of $\Gamma$ cannot lie in $\Gamma_e$ since a proper coloring of face of $\Gamma$ must use different colors on each of the vertices. This new vertex must either be $d - 1$ or a new vertex that does not belong to $[d]$. Note that the vertex set of $\Gamma \setminus [d]$ is larger than that of $\Gamma_e \setminus [d - 2]$ unless $f_0(\Delta) \le f_0(\lk_\Delta(e)) + 4 = f_0(\St_\Delta(e)) + 2$. \\

	It remains to show that $\dim S \le \frac{d}{2} - 1$, where $S$ is the dimension of the subcomplex $S$ defined as the (vertex) restriction of $\Gamma$ to faces that are disjoint from the unique facet $[d]$ of $\Gamma$. Given a face of $\Gamma$, we can partition its vertices into a face $F \in S$ and $G \in 2^{[d]}$. For the unique facet $[d]$ of $\Gamma$, the former is the empty set and the latter is all of $[d]$. Now consider the ridges ($(d - 2)$-dimensional faces) of $\Gamma$. We note that removing a single vertex yields a $(d - 3)$-dimensional face of $\Gamma$. If the vertex $p$ is removed from the first part, then we have $F \setminus p \in S_e$ for some edge $e \in \Delta$ and $|F \setminus p| \le \dim S_e + 1 \le \frac{d}{2} - 1 \Longrightarrow |F| \le \frac{d}{2}$. Suppose that the removed vertex $p$ is from the Boolean part. Then, we have that $F \in S_e$ for some edge $e \in \Delta$ and $|F| \le \dim S_e + 1 \le \frac{d}{2} - 2 + 1 = \frac{d}{2} - 1$. 

	\section{Connections to geometric Lefschetz decompositions} 

	The sizes of the faces from the disjoint vertex sets in the Boolean decomposition look similar to the exponents appearing in the Lefschetz decomposition induced by the multiplication by the usual $(1, 1)$-form in $d$-dimensional compact K\"ahler manifolds. We recall the statements below and make some comments on combinatorial applications. \\
	
	\begin{thm} \textbf{(Hard Lefschetz and Lefschetz decomposition, Theorem 14.1.1 on p. 237 -- 238 of \cite{Ara}, Theorem on p. 88 of \cite{CMSP}, p. 122 of \cite{GH}, Theorem 6.3 on p. 138 of \cite{Voi}) \\} \label{HL}
		\begin{enumerate}
			\item \textbf{(Hard Lefschetz) \\} 
				\begin{enumerate}
					\item Let $(X, \omega)$ be a compact K\"ahler manifold of complex dimension $D$. This induces a homomorphism \[ L : H^k(X) \longrightarrow H^{k + 2}(X) \] given by $[\alpha] \mapsto [\omega \wedge \alpha]$. The map \[ L^r : H^{ D - r }(X) \longrightarrow H^{ D + r }(X) \] is an isomorphism for $0 \le r \le D$. \\
					
					\item 
					A combinatorial application of this result is given on p. 77 -- 78 of \cite{St}. Note that the parameter $d$ we have been using corresponds to the complex dimension $D$. Given an integral simplicial $d$-polytope $P$, boundary $\Delta = \partial P$, and the associated toric variety $X_P$, the isomorphism $H^\cdot(X_P) \xrightarrow{\sim} A^\cdot(\Delta)$ halves degrees after an appropriate choice of linear system of parameters for the latter (Theorem 1.3 on p. 77 of \cite{St}). So, we are really considering maps \[ L : A^j(\Delta) \longrightarrow A^{j + 1}(\Delta) \] and \[ L^{2i} : A^{ \frac{d}{2} - i }(\Delta) \xrightarrow{\sim} A^{ \frac{d}{2} + i }(\Delta) \] associated to $r = 2i$ and $0 \le i \le \frac{d}{2}$. \\ 
				\end{enumerate}

			\item \textbf{(Lefschetz decomposition) \\}
				\begin{enumerate}
					\item We define by \[ P^i(X) \coloneq \ker( L^{D - i + 1} : H^i(X) \longrightarrow H^{2D - i + 2}(X) ) \] the \textbf{primitive cohomology} of $X$. Equivalently, we have \[ P^{D - r}(X) \coloneq \ker( L^{r + 1} : H^{D - r}(X) \longrightarrow H^{D + r + 2}(X) ). \]
					
					Then, we have that \[ H^m(X) = \bigoplus_{ k = 0 }^{ \lfloor \frac{m}{2} \rfloor } L^k P^{m - 2k}(X). \] This is called the \textbf{Lefschetz decomposition}. \\
					
					\item In the combinatorial context mentioned in Part 1 with the toric variety $X = X_P$ and $\Delta = \partial P$, this induces a decomposition \[ A^\ell(\Delta) = \bigoplus_{ k = 0 }^\ell L^k P^{\ell - k}(\Delta) \] after setting $m = 2\ell$ and halving degrees in the definition of the primitive decomposition. If $m = d$, we have $\ell = \frac{d}{2}$. \\ 
				\end{enumerate}

		\end{enumerate}
	\end{thm}
	
	Roughly speaking, the decomposition \[ \Gamma =   \{ F \cup G : F \in S, G \in 2^{[d - 2|F|]} \} \] has $(k - 1)$-dimensional faces of $S$ parametrizing (actions of) maps $L^k$ and the ``Boolean part'' $G \in 2^{[d - 2k]}$ (using a vertex set $[d]$ disjoint from $S$) giving elements of the primitive decomposition of degree $d - 2k$. The decomposition of $\Gamma$ here is similar to the decomposition \[ H^d(X) = \bigoplus_{k = 0}^{ \frac{d}{2} } L^k P^{d - 2k}(X) \] of $H^d(X)$ for a compact K\"ahler manifold $X$ of complex dimension $d$. It is interesting that the proposed decomposition of $\Gamma$ resembles the Lefschetz decomposition in the original geometric setting rather than the one in the usual application to $h$-vectors of simplicial complexes. \\
	
	In addition, the inductive construction we used indicates that a sort of (graded) Lefschetz map (which may depend on the degree) could involve ``paths of flag spheres'' or collections of flag spheres of a given dimension. Also, a single ``component'' of some objects corresponds to the $h$-vector of a single flag sphere of dimension $d$. If we use the same method, we would need to look at flag spheres which are dimension 2 higher than the ones considered before with the earlier spheres as links over edges and $h_{d + 2} \ne 0$ (so not something trivial like repeated suspensions). \\

	One example where Boolean lattices have been involved with Lefschetz properties in the literature is work of Kim--Rhoades \cite{KR} where on a sort of Lefschetz theory for certain exterior algebras is constructed. It involves a concatenation of two linear forms used to construct an analogue of a $(1, 1)$-form (p. 2915 of \cite{KR}). Other earlier examples of connections between Boolean lattices and Lefschetz properties mentioned there are the ``up operator'' involved in unimodality of the $q$-binomial coefficients (p. 2912 -- 2913 of \cite{KR}) and work of Hara--Watanabe \cite{HW} on the cohomology of products of projective lines and combinatorially constructed Lefschetz elements. Going back to geometry, an example where the cohomology of some combinatorially described object has a decomposition into a ``commutative'' part (a vector space) and an antisymmetric part from a torus is work on De Concini--Procesi on toric arrangements \cite{DCP}. Also, it would be interesting if we can relate these structures to algebraic interpretations of the top gamma vector for a large family of boundaries of flag simplicial polytopes as the signature of a certain toric variety in work of Leung--Reiner \cite{LR}. \\ 
	
	
	Somewhat coincidentally, the Boolean component in our setting is related to possible colorings and the usual linear system of parameters for balanced simplicial complexes seems to be connected to some kind of exterior algebra-like structure. \\

	\begin{prop}
		Let \[ \theta_i \coloneq \sum_{ \substack{ p \in V(\Gamma) \\ \kappa(p) = i } } x_p \] be the usual linear system of parameters for a balanced simplicial complex using a fixed proper coloring $\kappa : V(\Gamma) \longrightarrow [d]$ of the 1-skeleton of $\Gamma$ (Proposition 4.3 on p. 97 of \cite{St}). \\
		
		We have that $\theta_i^2 = 0$ in $B^\cdot(\Gamma)$ and the only nonzero monomials in $\theta_1, \ldots, \theta_d$ are the squarefree monomials $\theta_{i_1} \cdots \theta_{i_r}$ parametrized by $1 \le i_1 < \cdots < i_r \le d$. In particular, they correspond to $r$-element subsets of $[d]$. \\
	\end{prop}
	
	\begin{proof}
		Since no two vertices in the same face of $\Gamma$ have the same color, we have that \[ \theta_i^2 = \sum_{ \substack{ p \in V(\Gamma) \\ \kappa(p) = i } } x_p^2 \] in $k[\Gamma]$. Since squares of variables are quotiented out in $B^\cdot(\Gamma)$, we have that $\theta_i^2 = 0$ in $B^\cdot(\Gamma)$. \\
	\end{proof}


\end{document}